\documentclass[11pt]{article}
\usepackage[english]{babel}
\usepackage[utf8]{inputenc}
\usepackage[T1]{fontenc}
\usepackage{lmodern}
\usepackage{geometry}
\usepackage{amsthm,amsfonts}
\usepackage{amsmath}
\usepackage{graphicx}
\usepackage{tocloft}
\usepackage{dsfont}
\usepackage{pdfpages}
\usepackage{afterpage}
\usepackage{setspace}
\usepackage[nottoc,numbib]{tocbibind}
\allowdisplaybreaks
\author{\textsc{SIM\~AO CORREIA}}
\date{}
\title{ \Large\textbf{BLOWUP FOR THE NONLINEAR SCHRÖDINGER EQUATION WITH AN INHOMOGENEOUS DAMPING TERM IN THE $L^2$-CRITICAL CASE}}

\newtheoremstyle{mytheoremstyle} 
    {\topsep}                    
    {\topsep}                    
    {}                   
    {}                           
    {\scshape}                   
    {.}                          
    {.5em}                       
    {}  

\theoremstyle{mytheoremstyle} \newtheorem{nota}{Remark}
\theoremstyle{definition} 
\theoremstyle{plain}\newtheorem{teo}{Theorem}
\theoremstyle{plain} \newtheorem{lema}[teo]{Lemma}
\theoremstyle{plain} \newtheorem{cor}[teo]{Corollary}
\theoremstyle{plain} \newtheorem{prop}[teo]{Proposition}
\theoremstyle{mytheoremstyle} 
\newcommand{\real}{{\mathbb R}}

\newcommand{\complex}{{\mathbb C}}

\newcommand{\ub}{{\bar{u}}}

\newcommand \ben {\begin{equation}}
\newcommand \een {\end{equation}}
\newcommand \be {\begin{equation*}}
\newcommand \ee {\end{equation*}}
\newcommand \bi {\begin{itemize}}
\newcommand \ei {\end{itemize}}
\newcommand{\eum}{\epsilon_1}
\newcommand{\edois}{\epsilon_2}

\DeclareMathOperator*{\partere}{Re}
\DeclareMathOperator*{\parteim}{Im}

\begin{document}
\maketitle
\begin{abstract}
We consider the nonlinear Schrödinger equation with $L^2$-critical exponent and an inhomogeneous damping term. By using the tools developed by Merle and Raphael, we prove the existence of blowup phenomena in the energy space $H^1(\real)$.
\end{abstract}

\vskip10pt
\small{Keywords: Nonlinear Schr\"odinger equation; damping; blowup.
}
\section{Introduction}
$\indent$ We consider the Cauchy problem for the nonlinear Schrödinger equation
\ben\label{cauchy}
\left\{\begin{array}{llr}
iu_t + u_{xx} + |u|^{4}u +iau =0,& (t,x)\in[0,\infty)\times\real& \mbox{(NLS}_a\mbox{)}\\
u(0,x)=u_0(x),\ u_0:\real\to\complex
\end{array}\right.
\een
with a real inhomogeneous damping term $a\in C^1(\real;\real)\cap W^{1,\infty}(\real;\real)$. This is the one-dimensional $L^2$-critical case of the equation
\ben\label{geral}
iu_t + \Delta u + |u|^{p-1}u +iau =0,
\een
with $1<p<1+4/(N-2)$ if $N\ge3$ and $1<p<\infty$ if $N=1,2$. The equation (\ref{geral}) arises in several areas of nonlinear optics and plasma physics. The inhomogenous damping term corresponds to an electromagnetic wave absorved by an inhomogenous medium (cf. \cite{b1}, \cite{b2}).

It is known that, if $u_0\in H^1(\real^N)$, the Cauchy problem for (\ref{geral}) is locally well-posed (see Cazenave \cite{b3}, theorem 4.4.6). Moreover, if $T_a(u_0)$ is the maximal time of existence for the solution $u(t)$, one has the blowup alternative: if $T_a(u_0)<\infty$, then $\|\nabla u(t)\|_2\to\infty$ when $t\to T$.

The case where $a$ is constant was studied in \cite{darwich}, \cite{ohtatodorova}, \cite{tsutsumi}. In the supercritical case ($1+4/N<p <1 + 4/(N-2)$), for sufficiently small damping and special initial data with negative energy, the blowup of the solution is proved. The proof of this result is based on the variance method introduced in \cite{glassey} and \cite{zakharov}. Such method does not seem to work on the $L^2$-critical case for $a>0$. Also, for $1<p<1+4/(N-2)$ and for all initial data in $H^1(\real^N)$, one proves the global existence of the solutions for sufficiently large damping. The critical case ($p=1+4/N$) was studied in \cite{darwich}, where one proves, for small dimensions, the existence of blowup phenomena for small damping. The technique used therein is strongly based on the works of Merle and Raphaël (\cite{merleraphael1}, \cite{merleraphael2}).

Regarding the equation with inhomogeneous damping, it has been recently proved in \cite{diasfigueira} the existence of blow-up phenomena in the supercritical case, under similar conditions to those of the homogeneous case. Here, we shall consider the critical exponent $p=1+4/N$ and we prove the following result:
\begin{teo}
There exists $\delta>0$ such that, for $\|a\|_{W^{1,\infty}}<\delta$, there exists $u_0\in H^1(\real)$ such that the solution of (\ref{cauchy}) blows up in finite time.
\end{teo}
\begin{nota}
The result is stated in dimension one. We conjecture that it can be extended to higher dimensions (see \cite{merleraphael1}, \cite{darwich}).
\end{nota}
As a consequence of the technique used to prove the existence of blowup, we can prove an upper bound on the blow-up rate:
\begin{cor}
The explosive solution $u$ constructed in theorem 1 satisfies
$$
\|u_x\|_2\le C^*\frac{|\log(T-t)|^{1/4}}{\sqrt{T-t}}, \mbox{ for } t \mbox{ close to } T,
$$
where $C^*$ is an universal constant.
\end{cor}

As it was said before, the variance method does not seem to work in the critical case for the damped equation.  However, another method to prove the blow-up of
certain solutions of equation (1) in the case $a=0$ was introduced by
Merle and Raphael in \cite{merleraphael1}, based on the so-called {\it geometric
decomposition}
technique. The main goal of this method was to obtain an
upper bound on the blow-up rate, similar to the one presented in Corollary 2,
which was improved in \cite{merleraphael2} with a sharp upper bound estimate
(the log log upper bound). In \cite{darwich}, an extension of such a technique was made to the case where $a$ is a positive constant function, thus obtaining the first blow-up result for this critical case. Here, despite an inspiration on the arguments presented in \cite{darwich}, we do not follow the same steps. A simplification is made, to make the method used clearer. For example, in this proof, we shall not use Strichartz estimates, which were of particular importance in \cite{darwich}. One advantage is that the proof of blowup using this technique is done in a simpler way, which also implies a simplification of the proof of the upper bound on the blow-up rate. The disadvantage is that, while in \cite{darwich} one proves the log-log upper bound for the particular solution previously constructed, here we shall only prove the log upper bound.

We now recall some important invariances in the energy space $H^1(\real)$ for the nonlinear Schrödinger equation
\ben
iu_t + \Delta u + |u|^{4/N}u=0
\een
namely:
\bi
\item Mass (or charge): $C(u)=\|u\|^2_2=C(u_0)$;
\item Energy: $E(u)=\frac{1}{2}\|\nabla u\|_2^2 - \frac{1}{p+1}\|u\|_{p+1}^{p+1}=E(u_0)$;
\item Linear momentum: $M(u)=\parteim \int_{\real^N} \ub\nabla u=M(u_0)$;
\ei
For the (NLS$_a$) equation, these quantities are no longer conserved and one obtains the following evolution laws:
\bi
\item Mass evolution law:
\ben
\frac{d}{dt}C(u(t)) + 2\int_{\real^N} a(x)|u(t,x)|^2dx = 0;
\een
\item Energy evolution law:
\begin{align}
\frac{d}{dt}E(u(t))&= -\int_{\real^N}a(x)|\nabla u(t,x)|^2dx + \int_{\real^N}a(x)|u(t,x)|^{p+1}dx \\&\quad- \partere \int_{\real^N} \left(\nabla u(t,x)\cdot \nabla a(x)\right)\ub(t,x)dx;\nonumber
\end{align}
\item Linear momentum evolution law:
\ben
\frac{d}{dt}M(u(t)) + 2\int_{\real^N} a(x)\parteim\nabla u(t,x)\ub(t,x) dx = 0.
\een
\ei
Note that, from the mass evolution law, one has
\ben\label{controlomassa}
\|u_0\|_2\ e^{-\|a\|_\infty t}\le \|u(t)\|_2 \le \|u_0\|_2\ e^{\|a\|_\infty t}.
\een

The rest of this paper in organized as follows: in section 2, we will make a brief presentation of the technique used in \cite{merleraphael1}, highlighting the main steps. In section 3, a general idea of the proof is given, followed by its demonstration and, at the end, the log upper bound will be proved.

\begin{section}{The geometric decompositon method}
$\indent$ In this section, we shall consider the case where $a\equiv0$ and $N=1$,
\ben\tag{NLS}
iu_t+u_{xx} + |u|^4u=0.
\een

In this context, one may look for time-periodic solutions of the form $u(t,x)=e^{it}\phi(x)$. Inserting this expression in (NLS), we obtain the equation satisfied by $\phi$:
$$
-\phi_{xx}+\phi=|\phi|^4\phi.
$$
As proved in \cite{b3}, section 8.1, the above equation has non-trivial solutions in the energy space $H^1(\real)$. Futhermore, all the solutions are of the form 
$$
\phi(x)=e^{i\omega}Q(x-y),\ \omega,y\in\real,
$$
where $Q:\real\to\real$ is a positive decreasing radial function with exponential decay at infinity called the ground-state associated with (NLS). One may also prove (see \cite{b3}, section 8.4) that the ground-state is the only function (modulo translations and multiplication by a complex exponential) which minimizes the functional
$$
G(u)=\frac{\|\nabla u\|_2^2 \|u\|_2^{4}}{\|u\|_{6}^{6}}, \ u\neq 0.
$$
We define $Q_d=\frac{1}{2}Q+yQ_y$ and $Q_{dd}=\frac{1}{2}Q_d+y\left( Q_d\right)_y$. Moreover, we write the inner product in $L^2(\real)$ as $(\cdot,\cdot)$.

Consider a continuous function $u:[0,T]\to H^1(\real)$. From the variational characterization of the ground-state, it is proved in \cite{merleraphael1}, lemma 1, that, for small $\alpha$, if $0<\|u(t)\|_2^2-\|Q\|_2^2<\alpha$ and $E(u(t))\le \alpha\|u_x(t)\|_2^2$, there exist $C^1$ functions $x,\theta:[0,T]\to\real$ and $\lambda:[0,T]\to\real^+$ such that
\ben
\left| 1-\lambda(t)\frac{\| u_x(t)\|_2}{\|Q\|_2}\right| < \delta(\alpha)
\een
and
\ben
\|\lambda(t)^{1/2}e^{i\theta(t)}u(\lambda(t)(\cdot-x(t)),t)-Q\|_{H^1(\real)} < \delta(\alpha),
\een
where $\delta(\alpha)>0$ satisfies $\delta(\alpha)\to 0$ when $\alpha\to 0$. We define 
\ben\label{epsilon}
\epsilon(t)=\lambda(t)^{1/2}e^{i\theta(t)}u(\lambda(t)(\cdot-x(t)),t)-Q.
\een
The set of the functions $x,\theta,\lambda$ and $\epsilon$ is called a \textit{geometric decomposition} of $u$.

If $u_0\in H^1(\real)$ satisfies $0<\|u_0\|_2^2-\|Q\|_2^2<\alpha$ and $E(u_0)<0$, then, by the conservation of charge and energy, the corresponding solution  of (NLS) satisfies the conditions for the geometric decomposition. Therefore, one may write (NLS) using $\epsilon$ and the change  temporal variable
\be
s(t)=\int_0^t\frac{1}{\lambda^2(\tau)}d\tau,
\ee
thus obtaining the following system:
\begin{align}\label{eqe}
\partial_s\epsilon_1 - L_-\epsilon_2 &= \frac{\lambda_s}{\lambda}Q_d + \frac{x_s}{\lambda}Q_y + \frac{\lambda_s}{\lambda}(\epsilon_1)_d + \frac{x_s}{\lambda}(\epsilon_1)_y + \tilde{\theta_s}\epsilon_2 - R_2(\epsilon)\\\label{eqe2}
\partial_s\epsilon_2 + L_+\epsilon_1 &= -\tilde{\theta_s}Q - \tilde{\theta_s}\epsilon_1 + \frac{\lambda_s}{\lambda}(\epsilon_2)_d + \frac{x_s}{\lambda}(\epsilon_2)_y + R_1(\epsilon).
\end{align}
where $\tilde{\theta}_s=-1-\theta_s$, $\epsilon=\eum+i\edois$,  $L_+=-\Delta + 1 - 5Q^4$, $L_-=-\Delta + 1 - Q^4$ and $R_1(\epsilon), R_2(\epsilon)$ are formally quadratic in $\epsilon$. The operator $L=(L_+,L_-)$ is the linear operator close to the ground-state, which has been studied in \cite{weinstein}. Therein, there are proved the following identities:
$$
L_+(Q_d)=-2Q; \quad L_+(Q_y)=0; \quad L_-(Q)=0;
$$
$$
 L_-(yQ)=-2Q_y; \quad L_-(y^2Q)=-4Q_d.
$$
These properties are essential to make the estimates on the parameters of the geometric decompositon.

One can also write the linear momentum of $u$ as a function of $\epsilon$:
\ben\label{momento}
M(u(t)) = \frac{1}{\lambda}\left[\parteim\left(\int_\real \epsilon_y(t)\bar{\epsilon}(t)dx\right) - 2(\edois,Q_y)\right].
\een
By choosing $u_0\in H^1(\real)$ such that $M(u_0)=0$, then $M(u(t))\equiv 0$ and one has easily 
\ben
|(\edois, Q_y)(t)|\le \delta(\alpha)\|\epsilon(t)\|,
\een
 where $\|\epsilon\|^2=\int_\real |\epsilon_y|^2 + \int_\real |\epsilon|^2e^{-2^-|y|}$ and $2^-$ is a positive constant smaller than $2$ related with the properties of $Q$. This inequality is essential in the following results.

Now, if we choose a geometric decompositon (this choice can be made using the implicit function theorem) such that $\epsilon$ satisfies $(\epsilon_1,Q_d)=0;\ (\epsilon_2,Q_{dd})=0;\ (\epsilon_1,yQ)=0$ (the so-called ortogonality conditions), then it is possible to obtain
\ben\label{virialtransformado}
\left[ \left(1+\frac{1}{4\delta_0}(\epsilon_1,Q)\right)(\epsilon_2,Q_d)\right]_s\ge \delta_0\|\epsilon\|^2 + 2\lambda^2|E_0| - \frac{1}{\delta_0}(\epsilon_2,Q_d)^2.
\een
where $\delta_0$ is a positive constant.

By using the equations \eqref{eqe}, \eqref{eqe2} and the inequality (\ref{virialtransformado}), we derive the following result:
\begin{lema}
For $\alpha_1>0$ small, there exists $s_0>0$ such that
\be
(\epsilon_2,Q_d)(s)>0,\ \forall s>s_0.
\ee
Moreover, if $s_2>s_1>s_0$,
\ben\label{desigualdadeintegral}
3\int_{s_1}^{s_2}(\edois,Q_d) ds - C(\delta_0)\delta(\alpha)\le -\|yQ\|_2^2\log \frac{\lambda(s_2)}{\lambda(s_1)} \le 5\int_{s_1}^{s_2}(\edois,Q_d) ds + C(\delta_0)\delta(\alpha)
\een
and one has the quasi-monotony property
\ben\label{quasemonotonia}
\lambda(s_2)<2\lambda(s_1).
\een
\end{lema}
From the inequalities (\ref{virialtransformado}), (\ref{desigualdadeintegral}) and (\ref{quasemonotonia}), one proves that  $\lim_{t\to T_{max}} \|\nabla u(t)\|_2=\infty$ (or, equivalently, $\lim_{t\to T_{max}} \lambda(t)=0$), where $T_{max}$ is the maximal time of existence of the solution of (NLS). By using a refinement of the geometric decomposition, in which one introduces  $\tilde{\epsilon}=\epsilon+i\frac{(\edois,Q_d)}{\|yQ\|_2^2}W$, where $W=y^2Q+\nu Q$ and $\nu$ is such that $(W,Q_{dd})=0$, one obtains the following:
\begin{lema}\label{refinada}

1. There exist universal constants $\tilde{\delta}_0>0$ and $C>0$ such that, for small $\alpha>0$, there exists $\tilde{s}_1$ verifying
\ben\label{refinado}
\left[\left(1+\frac{(\eum, W_d)}{\|yQ\|_2^2}\right)(\edois,Q_d)\right]_s + C(\edois,Q_d)^4 \ge \tilde{\delta}_0\|\tilde{\epsilon}\|^2 + \lambda^2|E_0|, \ \forall \ s>\tilde{s}_1
\een
2. There exists a universal constant $B>0$ such that, for small $\alpha$, there exists $\tilde{s}_2>0$ such that
\ben\label{controloexp}
\lambda(s)^2\le \exp\left(-\frac{B}{(\edois,Q_d)^2(s)}\right), \forall \ s>\tilde{s}_2.
\een

\end{lema}

Finally, defining $t_k$ such that $\lambda(t_k)=2^{-k}$, it follows from the inequalities (\ref{desigualdadeintegral}), (\ref{quasemonotonia}) and (\ref{controloexp}) that $t_{k+1}-t_k\le C\lambda^2(t_k)|\log\lambda(t_k)|^{1/2}$, for large $k$. By summing in $k$, we deduce the finiteness of $T_{max}$ and the blowup is proved. The log upper bound is then a simple consequence of the above considerations.

\end{section}

\begin{section}{Proof of the main theorem}

$\indent$ The technique presented in the previous section works as long as it is possible to obtain the geometric decomposition for the solution of the equation one is working with. Unlike the (NLS) setting, since the mass and the energy are no longer conserved, one cannot guarantee \textit{a priori} that the solution of (NLS$_a$) is decomposable, even if the initial data satisfies the same conditions as before. Therefore, we shall work over certain uniformly bounded intervals contained in the maximal interval of existence of the solution of (NLS$_a$), where we know that it is possible to obtain the decomposition. The goal will be to prove that, by conveniently choosing the initial data and assuming $\|a\|_{W^{1,\infty}}$ small, then the largest of those intervals is actually the maximal interval of existence. Since those intervals are bounded uniformly, one has $T_a(u_0)<\infty$ and the blowup phenomenon is proved.

Let $u_0\in H^1(\real)$ with $E(u_0)<0$ and $M(u_0)=0$. Set $\alpha=2\left(\|u_0\|_2^2-\|Q\|_2^2\right)$, and assume that $\alpha>0$ is small. Therefore, on a small interval $[0,T_0]$, it is possible to decompose the solution geometrically. We denote $m=\lambda(0)$, parameter that has to be small for the following calculations. Futhermore, we suppose that $(\edois,Q_d)(0)>0$ and that
\ben
\lambda(0)^2\le e^{-\frac{B}{(\edois,Q_d)^2(0)}} \le \|\epsilon(0)\|^8.
\een
Notice that these conditions can be fulfilled: given $\tilde{u}_0\in H^1(\real)$ with negative energy and mass just above the critical mass $\|Q\|_2^2$, we consider the respective solution $\tilde{u}$ of (NLS). By the previous section, we know that the solution blows-up and that, for $t$ close to $T_{max}$, $(\edois, Q_d)(t)>0$ and $\lambda(t)^2\le \exp\left(-\frac{B}{(\edois,Q_d)^2(t)}\right)$ (cf. Lemmas 3,4). Now it is enough to consider $u_0=\tilde{u}(t)$, for a large fixed $t$.

In the following, we write $E(t):=E(u(t))$. Fixed $\alpha, m$ and $\|a\|_{W^{1,\infty}}$ small, we define the set $X$ as the set of all $T\ge 0$ such that\newline

(H1) $\quad T\le\frac{1}{2\|a\|_\infty}\log\frac{\|Q\|_2^2+\alpha}{\|Q\|_2^2+\frac{\alpha}{2}}$;\newline

(H2) $\quad E(t)\le\alpha\|u_x(t)\|^2_2,\ 0\le t\le T$;\newline

These two conditions and \eqref{controlomassa} allow us to obtain the geometric decomposition on the interval $[0,T]$. Now we define $k_0$ as the positive integer such that $\frac{1}{2^{k_0}}\ge\lambda(0)>\frac{1}{2^{k_0+1}}$ and $k_T$ as the integer such that $k_T\ge k_0$ and $\frac{1}{2^{k_T}}\ge\lambda(T)>\frac{1}{2^{k_T+1}}$.\newline

(H3) For each $k_T\ge k> k_0$, choose $t_k$ (taken in increasing order) such that
$\lambda(t_k)=2^{-k}$.
We also write $T=t_{k_T+1}$ and $0=t_{k_0}$. Then we require $t_{k+1}-t_k\le \lambda^{3/2}(t_k)$, $ k_T\ge k\ge k_0$;\newline

(H4) $\quad\lambda(\tilde{t})\le2\lambda(t)$, $\forall \tilde{t}, t:\ T\ge \tilde{t}\ge t\ge 0$;\newline

(H5) $\quad\lambda^{1/2}(t)\le \|\epsilon(t)\|^2$, $0\le t\le T$.\newline

It is important to notice that the hypothesis placed over the interval $[0,T]$ have a direct analogy with the lemmas from the previous section.

As a consequence of the broad inequalities and the continuity of the functions involved in the conditions (H1)-(H5), the set $X$ is closed in $[0,T_a(u_0))$. Since $0\in X$, $X$ is nonempty. If one proves that $X$ is open in $[0, T_a(u_0))$, then one obtains $X=[0,T_a(u_0))$. Since $X$ is bounded (by (H1)), this proves finite-time blowup. To show that $X$ is open in $[0, T_a(u_0))$, we shall prove that, if $T\in X$, then, on the interval $[0,T]$, one verifies stronger conditions than those that define the set $X$. By continuity, this implies that, for small $\delta>0$, $T+\delta\in X$, and $X$ is open.

\begin{nota}
For small $\|a\|_\infty$,
\be
 T< 2 \le\frac{1}{2\|a\|_\infty}\log\frac{\|Q\|_2^2+\alpha}{\|Q\|_2^2+\frac{\alpha}{2}}.
\ee
In fact, using the lenght hypothesis for the intervals $[t_k,t_{k+1}]$,
\be
T=\sum_{k_0}^{k_T+1} t_{k+1}-t_k \le \sum_{k_0}^{k_T+1}\left(\frac{1}{2\sqrt{2}}\right)^k< 2
\ee
\end{nota}

\begin{lema}
For any $\delta>0$, there exists $a_0>0$ such that, for $0<\|a\|_{W^{1,\infty}}<a_0$, one has, over the interval $[0,T]$,
\be
E(t)\lambda^{3/2}(t)<\delta.
\ee
\end{lema}
\begin{proof}
We shall prove that, for each $i\ge k_0$, if $E(t_j)\lambda^{3/2}(t_j)<\delta\left(1-\frac{1}{j}\right),\ \forall j\le i$, then
\be
E(t)\lambda^{3/2}(t) < \delta\left(1-\frac{1}{i+1}\right), \ \forall \ t\in[t_i,t_{i+1}].
\ee
The result then follows by induction.
Suppose that
\be
E(t_j)\lambda^{3/2}(t_j)<\delta\left(1-\frac{1}{j}\right), \forall j\le i
\ee
Recall that, over the interval $[t_i,t_{i+1}]$, $\lambda(t)>2^{-i+2}$ and, by the geometric decomposition, $\lambda(t)$ is approximately $\|Q\|_2\|u_x(t)\|_2^{-1}$.
Using the energy evolution law,
$$
\frac{dE(t)}{dt}\le -\|a\|_\infty(p+1)E(t) + \|a\|_\infty\frac{p+3}{2}\| u_x(t)\|_2^2 + \|a_x\|_\infty\|u_x(t)\|_2e^{\|a\|_\infty t}\|u_0\|_2
$$
and so
\begin{align}\label{diferencial}
\frac{d}{dt}\left(e^{\|a\|_\infty(p+1)t}E(t)\right) &\le \|a\|_\infty\frac{p+3}{2}e^{\|a\|_\infty(p+1)t}\| u_x(t)\|_2^2 \nonumber\\&\quad+ \| a_x\|_\infty\|u_x(t)\|_2e^{\|a\|_\infty(p+2) t}\|u_0\|_2.
\end{align}
If $i\neq k_0+1$, integrating over the interval $[t_{i-1},t]$ with $t_i<t\le t_{i+1}$,
\begin{align*}
e^{\|a\|_\infty(p+1)t}E(t) - e^{\|a\|_\infty(p+1)t_{i-1}}E(t_{i-1})&\le \frac{p+3}{2(p+1)}2^{2(i+1)}\left(e^{\|a\|_\infty(p+1)t}-e^{\|a\|_\infty(p+1)t_{i-1}}\right) \\
&+ \frac{\| a_x\|_\infty}{\|a\|_\infty(p+2)}2^{i+1}\left(e^{\|a\|_\infty(p+2)t}-e^{\|a\|_\infty(p+2)t_{i-1}}\right),
\end{align*}
which implies
\begin{align*}
E(t)&\le E(t_{i-1})e^{-\|a\|_\infty(p+1)(t-t_{i-1})} +  \frac{p+3}{2(p+1)}2^{2(i+1)}\left(1-e^{-\|a\|_\infty(p+1)(t-t_{i-1})}\right)\\
&\quad+ \frac{\| a_x\|_\infty}{\|a\|_\infty(p+2)}2^{i+1}\left(e^{\|a\|_\infty t} - e^{-\|a\|_\infty(p+1)(t-t_{i-1})}e^{\|a\|_\infty t_{i-1}}\right).
\end{align*}
Multiplying by $\lambda^{3/2}(t)<\lambda^{3/2}(t_{i-1})=2^{-\frac{3}{2}(i-1)}$ and using the induction hypothesis,
\begin{align*}
E(t)\lambda^{3/2}(t)&\le \delta\left(1-\frac{1}{i-1}\right) +  C2^{i/2}\left(1-e^{-\|a\|_\infty(p+1)(t-t_{i-1})}\right) \\&\quad+ \frac{\| a_x\|_\infty}{\|a\|_\infty(p+2)}2^{-i/2}\left(e^{\|a\|_\infty t} - e^{-\|a\|_\infty(p+1)(t-t_{i-1})}e^{\|a\|_\infty t_{i-1}}\right).
\end{align*}
From the interval lenght hypothesis, one has $t_{i+1}-t_{i-1}\le  2^{-i+2}$. Since
\begin{align*}
&\frac{1}{\|a\|_\infty}2^{-i/2}\left(e^{\|a\|_\infty t} - e^{-\|a\|_\infty(p+1)(t-t_i)}e^{\|a\|_\infty t_i}\right)\\\le\ & e^{\|a\|_\infty T}2^{-3i/2}\left(\frac{e^{\|a\|_\infty 2^{-i+2}} - e^{-\|a\|_\infty(p+1)2^{-i}+2}}{\|a\|_\infty 2^{-i}}\right) \le K2^{-3i/2},
\end{align*}
with $K$ independent of $i$ and $\|a\|_\infty$, we deduce
\begin{align*}
E(t)\lambda^{3/2}(t)\le \delta\left(1-\frac{1}{i}\right) +  C2^{i/2}\left(1-e^{-\|a\|_\infty(p+1)2^{-i+2}}\right) + \frac{\| a_x\|_\infty}{(p+2)}K2^{-3i/2}.
\end{align*}
It now suffices to check that, independently of $i$, for small $\|a\|_{W^{1,\infty}}$, 
$$
C2^{i/2}\left(1-e^{-\|a\|_\infty(p+1)2^{-i}}\right) + \frac{\| a_x\|_\infty}{(p+2)}K2^{-3i/2}\le \delta\left(\frac{1}{i}-\frac{1}{i+1}\right).
$$
For the case $i=k_0+1$, we integrate (\ref{diferencial}) over the interval $[t_{k_0}, t]$, with $t_{k_0}<t\le t_{k_0+1}$ and we use the fact that $E(t_{k_0})<0$.
\end{proof}

The following lemma solves the problem of the non-conservation of the linear momentum:
\begin{lema}
For small $\|a\|_\infty$ and $\alpha$, one has $|(\edois,Q_y)(t)|\le 2\delta(\alpha)\|\epsilon(t)\|$, $\forall t\in [0,T]$.
\end{lema}
\begin{proof}
For each $k_0\le k\le k_T$ and $t\in[t_k,t_{k+1}]$,
\begin{align*}
&\left|\int_\real \parteim u_x(t)\ub(t) dx - \int_\real \parteim u_x(t_k)\ub(t_k) dx\right|=2\left|\int_{t_k}^{t}\int_\real a \parteim u_x(s,x)\ub(s,x)dx ds\right|  \\
\le&\ 2\|a\|_\infty\|u_0\|_2e^{\|a\|_\infty T}\int_{t_k}^{t} \|\nabla u(s)\|_2ds \le
 C\|a\|_\infty\|u_0\|_2  \int_{t_k}^{t_{k+1}} \frac{1}{\lambda(s)} ds\\\le&\ C\|a\|_\infty\|u_0\|_2 \frac{2}{\lambda(t_{k+1})} (t_{k+1}-t_k)\\
\le&\ C'\|a\|_\infty\|u_0\|_2\lambda^{1/2}(t_k).
\end{align*}
Given $t\in[0,T]$, let $k_t$ be such that $t\in [t_{k_t}, t_{k_t+1}]$. Then
\begin{align*}
\left|\int_\real \parteim u_x(t)\ub(t) dx\right| &\le \left|\int_\real \parteim u_x(t)\ub(t) dx - \int_\real \parteim u_x(t_k)\ub(t_k) dx\right|\\&\quad+ \sum_{i=0}^{k_t-1}\left|  \int_\real \parteim u_x(t_{i+1})\ub(t_{i+1}) dx - \int_\real \parteim u_x(t_i)\ub(t_i) \right|\\&\le C'\|a\|_\infty\|u_0\|_2\sum_{i=0}^{k_t}\lambda^{1/2}(t_i)\le C'\|a\|_\infty\|u_0\|_2\sum_{i=0}^\infty \left(\frac{1}{\sqrt{2}}\right)^i.
\end{align*}
Recalling the last property of the interval $[0,T]$, (H5), we obtain
$$
\left|\lambda(t)\int_\real \parteim u_x(t)\ub(t) dx\right|\le \|\epsilon(t)\|^4\left(C'\|a\|_\infty\|u_0\|_2\sum_{i=0}^\infty \left(\frac{1}{\sqrt{2}}\right)^i\right)\le \delta(\alpha)\|\epsilon(t)\|,
$$
for small $\|a\|_\infty$ and $\alpha$.
Using \eqref{momento}, we deduce finally
\begin{align*}
|(\edois,Q_y)(t)|&\le \left|\parteim\left(\int_\real \epsilon_y(t)\bar{\epsilon}(t)dx\right)\right| + \left|\lambda(t)M(u(t))\right|\le \|\epsilon(t)\|_2\|\epsilon_y(t)\|_2+\delta(\alpha)\|\epsilon(t)\| \\&\le 2\delta(\alpha)\|\epsilon(t)\|.
\end{align*}

\end{proof}
Let us introduce a new time variable
\be
s(t)=\int_0^t \frac{1}{\lambda^2(\tau)}d\tau
\ee
and define $S=s(T)$ and $s_k=s(t_k)$, $ k_0\le k \le k_T$. Then, from the expression of $\epsilon$, \eqref{epsilon}, we may write (NLS$_a$) in terms of $\epsilon=\epsilon_1 + i\epsilon_2$ over the interval $[0,S]$:

\begin{align}\label{aeqeum}
\partial_s\epsilon_1 - L_-\epsilon_2 &= \frac{\lambda_s}{\lambda}Q_d + \frac{x_s}{\lambda}Q_y + \frac{\lambda_s}{\lambda}(\epsilon_1)_d + \frac{x_2}{\lambda}(\epsilon_1)_y + \tilde{\theta_s}\epsilon_2 - R_2(\epsilon) -a\lambda^2\epsilon_1
\\
\label{aeqedois}
\partial_s\epsilon_2 + L_+\epsilon_1 &= -\tilde{\theta_s}Q - \tilde{\theta_s}\epsilon_1 + \frac{\lambda_s}{\lambda}(\epsilon_2)_d + \frac{x_s}{\lambda}(\epsilon_2)_y + R_1(\epsilon)-a\lambda^2\epsilon_2.
\end{align}

Through the control of $|(\edois,Q_y)|$ given by the previous lemma and the same ortogonality conditions as the last section, one has the following (see \cite{merleraphael1}, proposition 1):

\begin{lema}\label{asem(e1,Q)}
There exists an universal constant $\delta_0>0$ such that, for $~\alpha$ and $\|a\|_{W^{1,\infty}}$ small,
\begin{align*}
\left[\left(1+\frac{1}{4\delta_0}(\epsilon_1,Q)\right)(\epsilon_2,Q_d)\right]_s\ge&\ \delta_0\|\epsilon\|^2 - 2 \lambda^2E - \frac{1}{\delta_0}(\epsilon_2,Q_d)^2 \\&- \|a\|_\infty\lambda^2|(\epsilon_2,Q_d)\left((\epsilon_2,Q_d)+2(\epsilon_1,Q)\right)|
\end{align*}
\end{lema}

To prove such a result, several steps are needed: first, one calculates $(\edois, Q_d)_s$, use the first two ortogonality conditions and the energy expression in terms of $\epsilon$ to obtain
$$
(\edois, Q_d)_s \ge H(\epsilon, \epsilon) -2\lambda^2E - \frac{x_s}{\lambda}(\edois, (Q_d)_y) + G(\epsilon) - \lambda^2(a\edois, Q_d),
$$
where $H$ is some bilinear form related to $(L_-, L_+)$ and $G$ is a higher-order remainder. The last ortogonality condition guarantees that $|x_s/\lambda|\le C\delta(\alpha)\|\epsilon\|$. A precise study of the bilinear form $H$ insures that, in the subspace where $(\edois,Q_d)=(\eum,Q)=(\eum, yQ)=(\edois, Q_{dd})=0$, the form is coercive. Using this information, one obtains the following intermediate inequality
$$
(\edois, Q_d)_s\ge \tilde{\delta}_0\|\epsilon\|^2-2\lambda^2E - \frac{4}{\tilde{\delta}_0}((\eum,Q)^2 + (\edois,Q_d)^2) - \|a\|_\infty\lambda^2|(\edois, Q_d)|.
$$
Finally, one proves that $(\eum,Q)^2$ is controled by $((\eum, Q)(\edois, Q_d))_s$ and obtains the final inequality.
%
%
\begin{nota}\label{nota}
Due to the hypothesis over the interval $[0,T]$, it is possible to simplify the previous inequality:
\begin{enumerate}
\item Since $2E\lambda^2\le 2E\lambda^{3/2}\lambda^{1/2}\le \frac{\delta_0}{2}\|\epsilon\|^2$ over $[0,T]$, we obtain
$$
\delta_0\|\epsilon\|^2 - 2\lambda^2E\ge \frac{\delta_0}{2}\|\epsilon\|^2;
$$
\item On the other hand, using (H4), $\lambda(t)\le 2\lambda(0)=2m, \forall\ t\in[0,T]$. Therefore $\lambda$ is bounded on $[0,T]$ by a constant $L$ that only depends on $m$ and, for small $\|a\|_\infty$,
$$
\left|\|a\|_\infty\lambda^2(\edois,Q_d)((\edois,Q_d)+2(\eum,Q))\right|\le \|a\|_\infty L^2C\|\epsilon\|^2\le \frac{\delta_0}{4}\|\epsilon\|^2.
$$
\end{enumerate}
In this way, one obtains the following inequality:
\ben\label{semenergia}
\left[\left(1+\frac{1}{4\delta_0}(\epsilon_1,Q)\right)(\epsilon_2,Q_d)\right]_s\ge \frac{\delta_0}{4}\|\epsilon\|^2 - \frac{1}{\delta_0}(\epsilon_2,Q_d)^2.
\een
\end{nota}

We now turn to the inequality analogous to (\ref{desigualdadeintegral}). The terms associated to the damping parameter turn out to be irrelevant, since their integral over the set $[0,S]$ is bounded by a function of $\|a\|_\infty$ which converges to 0 when $\|a\|_\infty\to 0$.

\begin{lema}
For small $a$ and $\alpha$, one has, over the interval $[0,S]$, $(\edois, Q_d)>0$ and 
\ben\label{desigualdadeintegralamort}
3\int_{s_1}^{s_2}(\edois,Q_d) ds - C(\delta_0)\delta(\alpha)\le -\|yQ\|_2^2\log \frac{\lambda(s_2)}{\lambda(s_1)} \le 5\int_{s_1}^{s_2}(\edois,Q_d) ds + C(\delta_0)\delta(\alpha).
\een
\end{lema}
\begin{proof}
Since $(\edois,Q_d)(0)>0$, it is enough to check that, if $(\edois,Q_d)=0$, then $(\edois,Q_d)_s>0$. If there exists $s\ge0$ such that $(\edois,Q_d)=0$ and $(\edois,Q_d)_s\le0$, then, by (\ref{semenergia}), 
\be
\|\epsilon\|^2\le 0,
\ee
which is absurd. Therefore $(\edois,Q_d)>0$ on $[0,S]$.
To obtain the integral inequality, we proceed as in \cite{merleraphael1}. The problem is controlling the terms associated with $a$. For example, by taking the $L^2$ inner product of (\ref{aeqeum}) with $y^2Q$ and integrating, one obtains the term
\be
\int_{s_1}^{s_2} \|a\|_\infty \lambda^2(\eum,y^2Q) ds.
\ee
However, simply notice that
\be
\left|\int_{s_1}^{s_2} \|a\|_\infty \lambda^2(\eum,y^2Q) ds\right| \le \delta(\alpha)\|a\|_\infty\int_0^S \lambda^2(s)ds = \delta(\alpha)\|a\|_\infty T\le 2\delta(\alpha)\|a\|_\infty.
\ee
Therefore, for small $\|a\|_\infty$, we deduce $\left|\int_{s_1}^{s_2} a \lambda^2(\eum,y^2Q) ds\right|\le \delta(\alpha)$. The remainder terms are controlled in a similar way.
\end{proof}

Using the previous result, we prove a stronger quasi-monotony property than the one in the definition of $X$:
\begin{lema}
For small $\alpha$,
\be
\lambda(\tilde{t})<\frac{3}{2}\lambda(t), \quad T\ge\tilde{t}\ge t\ge0.
\ee
\end{lema}
\begin{proof}
If such an inequality was not true for some $t_1<t_2$, then, by (\ref{desigualdadeintegralamort})
$$
\|yQ\|_2^2\log\frac{3}{2} - C(\delta_0)\delta(\alpha)\le \|yQ\|_2^2\log\frac{\lambda(s_2)}{\lambda(s_1)} - C(\delta_0)\delta(\alpha) \le -3\int_{s_1}^{s_2}(\edois, Q_d) ds <0, 
$$
which is absurd, for small enough $\alpha$.
\end{proof}

Since the term $a\lambda^2$ is bounded by a small constant, one may apply a reasoning similar to remark 3.2 to prove a result completely analogous to the first part of lemma \ref{refinada}:
\begin{lema}
There exist universal constants $\tilde{\delta}_0>0$ e $C>0$ such that, for $m$ and $\alpha$ small,
\ben\label{refinadoamort}
\left[\left(1+\frac{(\eum, W_d)}{\|yQ\|_2^2}\right)(\edois,Q_d)\right]_s + C(\edois,Q_d)^4 \ge \tilde{\delta}_0\|\tilde{\epsilon}\|^2 - \lambda^2E, \ \forall \ s>0
\een
\end{lema}

Now we prove the following
\begin{prop}
There exist universal constants $B', \sigma>0$ such that, for small $\alpha$ and $m$, 
\ben\label{controloexpamort}
\lambda(0)^{2\sigma}\lambda(s)^2\le \exp\left(-\frac{B}{(\edois,Q_d)^2(s)}\right), \  \ 0\le s\le S.
\een
\end{prop}
\begin{nota}
The above inequality is equivalent to
\ben\label{maiorlog}
(\edois,Q_d)(s)\ge \frac{B^*}{|\log \left(\lambda(0)^\sigma\lambda(s)\right)|^{1/2}}.
\een
\end{nota}
\begin{proof}
What follows is an adaptation of the proof for the proposition 8 in \cite{merleraphael1}.
Define 
\ben\label{expressaof}
f(s)=\left(1+\frac{(\eum,W_d)}{\|yQ\|_2^2}\right)(\edois,Q_d). 
\een
For small $\alpha>0$,
\be
\frac{1}{2}(\edois,Q_d)\le f \le 2(\edois,Q_d).
\ee
Then $f>0$ for $s\in[0,S]$ and, using (\ref{refinadoamort}), there exists a universal constant $C'>0$ such that (see remark \ref{nota}.1)
\be
f_s+C'f^4\ge0.
\ee
Integrating this inequality over $[0,s]$, we obtain
\be
\frac{1}{f^3(s)}\le C's+\frac{1}{f^3(0)},
\ee
and from \eqref{expressaof}, we deduce
\ben\label{scuboamort}
(\edois,Q_d)(s)\ge \frac{1}{2\left(C's+\frac{1}{f^3(0)}\right)^{1/3}},\ \forall s>0.
\een
Now, from (\ref{desigualdadeintegralamort}) and \eqref{scuboamort}, we obtain
\be
3\int_0^s (\edois,Q_d)ds \le -\|yQ\|_2^2\log\frac{\lambda(s)}{\lambda(0)} + C(\delta_0)\delta(\alpha)\le -\frac{\|yQ\|_2^2}{2}\log\frac{\lambda(s)}{\lambda(0)}
\ee
and
\be
C''\left(\left(C's+\frac{1}{f^3(0)}\right)^{2/3}-\frac{1}{f^2(0)}\right)\le -\log\frac{\lambda(s)}{\lambda(0)}.
\ee
Hence, 
\ben\label{loglambda}
\frac{C''}{4(\edois,Q_d)^2}\le -\log \lambda(s) + \log \lambda(0) + \frac{C''}{f^2(0)}.
\een
Since $\lambda(0)\le e^{-\frac{B}{(\edois,Q_d)^2(0)}}$, there exists $\sigma>0$ universal constant such that
$$
 \frac{C''}{f^2(0)}\le -(\sigma+1)\log\lambda(0)
$$
and from \eqref{loglambda},
$$
\frac{C''}{4(\edois,Q_d)^2}\le -\log \left(\lambda(0)^\sigma\lambda(s)\right).
$$
Therefore, there exists an universal constant $B'>0$ such that
\be
- \log \left(\lambda(0)^{2\sigma}\lambda^2(s)\right)\ge -2\log \left(\lambda(0)^\sigma\lambda(s)\right) \ge \frac{B'}{(\edois,Q_d)^2(s)}, 
\ee
or, equivalently,
$$
\lambda(0)^{2\sigma}\lambda(s)^2\le \exp\left(-\frac{B'}{(\edois,Q_d)^2(s)}\right), \forall\ s\in [0,S].
$$
\end{proof}
\begin{lema}
There exists an universal constant $D$ such that, for each $k_0\le k\le k_T$,
$$
t_{k+1}-t_k\le D\left|\log\left(\lambda(0)^\sigma\lambda(t_k)\right)\right|^{1/2}\lambda^2(t_k)
$$
\end{lema}
\begin{proof}
First, using the quasi-monotonicity property and (\ref{desigualdadeintegralamort}),
$$2\|yQ\|_2^2\log 2\ge \|yQ\|_2^2 + C(\delta_0)\delta(\alpha) \ge 3\int_{s_k}^{s_{k+1}} (\edois, Q_d)(s)ds.
$$
Now, from (\ref{maiorlog}),
\begin{align*}
\int_{s_k}^{s_{k+1}} (\edois, Q_d)(s)ds &\ge B^*\int_{s_k}^{s_{k+1}} \frac{1}{|\log(\lambda(0)^{\sigma}\lambda(s)|^{1/2}}ds  \ge  B^*\int_{t_k}^{t_{k+1}} \frac{1}{\lambda^2(t)|\log(\lambda(0)^{\sigma}\lambda(t)|^{1/2}}dt \\ &\ge \frac{t_{k+1}-t_k}{4\lambda^2(t_k)|\log(\lambda(0)^{\sigma}\lambda(t)|^{1/2}}.
\end{align*}
The result follows from combining the two above inequalities.
\end{proof}

\textit{Proof of theorem 1}.
Since $X$ is nonempty and closed in $[0,T_a(u_0))$, $X=[0,T_a(u_0))$ iff $X$ open in $[0,T_a(u_0))$. Let $T\in X$ be arbitrary. Joining the conclusions of remark 2 and lemmas 5, 6, 9, 12 and proposition 11, one has, for small $\alpha, m$ and $\|a\|_{W^{1,\infty}}$  (chosen by this order) \footnote{Note that this choice is independent of $T$.}\newline

(\~{H}1) $T<2<\frac{1}{2\|a\|_\infty}\log\frac{\|Q\|_2^2+\alpha}{\|Q\|_2^2+\frac{\alpha}{2}}$; \newline

(\~{H}2) $E(t)<\delta(\alpha)\|u_x(t)\|^{3/2}_2< \delta(\alpha)\|u_x(t)\|^{2}_2$; \newline

(\~{H}3) For each $k_T\ge k\ge k_0$, $\ t_{k+1}-t_k\le D\left|\log\left(\lambda(0)^\sigma \lambda(t_k)\right)\right|^{1/2}\lambda^2(t_k)< \lambda^{7/4}(t_k)$; \newline

(\~{H}4) $\lambda(\tilde{t})\le \frac{3}{2}\lambda(t)<2\lambda(t)$, $\forall  \tilde{t}, t: \ T\ge \tilde{t}\ge t\ge 0$; \newline

(\~{H}5) $\lambda^{1/2}(t)\le\exp\left(-\frac{B}{4(\edois,Q_d)^2(t)}\right) < \frac{1}{2}\|\epsilon(t)\|^2$.\newline

One now applies a standart bootstrap argument since, in a neighbourhood of $T$, one has stronger conditions than those defining the set $X$. Then $X$ is open and $X=[0,T_a(u_0))$. From the definition of $X$, $T_a(u_0)\le \frac{1}{2\|a\|_\infty}\log\frac{\|Q\|_2^2+\alpha}{\|Q\|_2^2+\frac{\alpha}{2}}$, which proves finite-time blowup.
\hfill$\qed$
\vskip10pt
\textit{Proof of corollary 2.}
For the sake of simplicity, we write $T=T_a(u_0)$. Since the solution blows-up in finite time, we may define, for each $k\ge k_0$ $t_k\in [0,T)$ such that $t_k\to T$ and $\lambda(t_k)=2^{-k}$. By the previous proof, $X=[0,T)$, and so, by (\~H3),
$$
 t_{k+1}-t_k\le D\left|\log\left(\lambda(0)^\sigma \lambda(t_k)\right)\right|^{1/2}\lambda^2(t_k), \ k\ge k_0.
$$
Then, for $k$ large,
 $$
 t_{k+1}-t_k\le C\left|\log\left(\lambda(t_k)\right)\right|^{1/2}\lambda^2(t_k).
$$
Fix $n$ large. Summing in $k\ge n$,
\begin{align*}
T-t_n &\le C\sum_{k\ge n} 2^{-2k}\sqrt{k}=C \sum_{n\le k < 2n} 2^{-2k}\sqrt{k} + C\sum_{k\ge 2n} 2^{-2k}\sqrt{k}&\\
&\le C2^{-2n}\sqrt{n} + C\sum_{j\ge0} 2^{-2(j+2n)}\sqrt{2n+j} &\\&\le C2^{-2n}\sqrt{n} + C2^{-4n}\sqrt{n}\sum_{j\ge0} 2^{-2j}\sqrt{2+\frac{j}{n}}&\\
&\le C2^{-2n}\sqrt{n} +C2^{-4n}\sqrt{n}\le C2^{-2n}\sqrt{n} = C\lambda^2(t_n)|\log\lambda(t_n)|^{1/2}.
\end{align*}
Given $t$ close to $T_a(u_0)$, $t\in[t_n,t_{n+1}]$ for some large $n$. Therefore, by (H4),
\ben\label{log1}
\lambda^2(t)\left|\log\frac{\lambda(t)}{2}\right|^{1/2} \ge C\lambda^2(t_n)|\log\lambda(t_n)|^{1/2}\ge C(T-t_n).
\een
Set $g(x)=x^2|\log\frac{x}{2}|^{1/2}$. For $t$ close to $T$ and $C^*=\sqrt{C}$,
\begin{align}\label{log2}
g\left(\frac{C^*\sqrt{T-t}}{|\log(T-t)|^{1/4}}\right)&=\frac{C(T-t)}{|\log(T-t)|^{1/2}}\left|\log\frac{C^*\sqrt{T-t}}{2|\log (T-t)|^{1/4}}\right|^{1/2}\nonumber\\&
=\frac{C(T-t)}{|\log(T-t)|^{1/2}}\frac{1}{\sqrt{2}}\left|\log(T-t)-\log\left(\frac{4}{C}|\log (T-t)|^{1/2}\right)\right|^{1/2}\nonumber\\
&=\frac{C}{\sqrt{2}}(T-t)\left|1-\frac{\log\left(\frac{4}{C}|\log (T-t)|^{1/2}\right)}{\log(T-t)}\right|^{1/2}
\le C(T-t).
\end{align}
Since $g$ is nondecreasing in a neighbourhood of $0$, by \eqref{log1} and \eqref{log2}, one has, for $t$ close to $T$,
\be
\lambda(t)\ge \frac{C^*\sqrt{T-t}}{|\log(T-t)|^{1/4}},
\ee
which concludes this proof.
%
%
%
%

\end{section}
\begin{section}{Acknowledgements}

$\indent$ I thank the Fundação Calouste Gulbenkian for the financial support, Darwich Mohamad for his availability to discuss his work and Thierry Cazenave for important improvements. Finally, I am grateful to M\'{a}rio Figueira, who suggested this problem and offered multiple and interesting points of view.
\end{section}

Adress: CMAF/UL, Av.\ Prof.\ Gama Pinto 2, 1649-003 Lisboa, Portugal

E-mail: simaofc@campus.ul.pt
\end{document}